\documentclass[a4,11pt]{amsart}
\usepackage{a4wide} 
\usepackage{amssymb,latexsym,color}
\usepackage{amsmath,amssymb}
\usepackage[all]{xy}

\newcommand\rres{\mathbin{\backslash}}
\newcommand\lres{\mathbin{/}}
\newcommand\RRS{{\sf RRS}}
\newcommand\RS{{\sf RS}}

\newtheorem{theorem}{Theorem}[section]

\newtheorem{rem}[theorem]{Remark}
\newtheorem{problem}[theorem]{Problem}

\newtheorem{defin}[theorem]{Definition}
\newenvironment{definition}{\begin{defin}\rm}{\end{defin}}

\title[FMP for RS]{Finite Model Properties for Residuated Semigroups}

\author{Szabolcs Mikul\'as}
\email{miksza001@gmail.com}
\date{21/11/2022}

\begin{document}

\maketitle

\begin{abstract}
We have a quick look at various finite model properties for residuated semigroups.
In particular, we solve Problem~19.17 from the monograph by Hirsch and Hodkinson~\cite{HH-rel-02}.
%We show that the class of representable residuated semigroups lacks the finite representation property.
%\\
%{\sc Keywords:} finite representation, Lambek Calculus, residuated semigroups
\end{abstract}

%\section{Introduction}

\section{Introduction}\label{sec-int}

Residuated semigroups have been investigated on their own right, but also in connection with susbstructural logics, like the 
Lambek Calculus (LC) Lambek~\cite{lam-mat-58}, and relevance logics Anderson et al.~\cite{ABD-ant-92}.
Indeed, relational semantics, i.e.\ models based on binary relations, have been proposed for these substructural logics, see
van Benthem~\cite{ben-lan-91} and Routley and Meyer~\cite{RM-sem-73},
and the algebraic settings of these relational semantics give rise to residuated semigroups (and their expansions to larger similarity types).
For instance, representable residuated semigroups (see below) provide sound and complete semantics for LC, 
see Andr\'eka and Mikul\'as~\cite{AM-lam-94}.  
The connection between (fragments of) relevance logics and (variants of) residuated semigroups have been investigated
in e.g.\ Bimb\'o et al.~\cite{BDM-rel-09}, Maddux~\cite{mad-rel-10}, Hirsch and Mikul\'as~\cite{HM-pos-11} and Mikul\'as~\cite{mik-alg-09},\cite{mik-low-15}.

In this note, we will concentrate on representable residuated semigroups with finite bases (see below) and show that they have smaller expressive power
than representable semigroups in general. In particular, we show that there is a finite representable residuated semigroup that is not isomorphic to any
representable residuated semigroup with a finite base, thus providing a solution to Problem~19.17 of Hirsch and Hodkinson~\cite{HH-rel-02}.

Next we define (representable) residuated semigroups.
Strictly speaking, these are not algebras but algebraic structures (because of the ordering $\le$) but the reader should not have any problem with applying the usual algebraic techniques (like taking subalgebras) to algebraic structures.

\begin{definition}\label{def:rs}
A \emph{residuated semigroup} ${\mathcal A}=(A,\le,\circ, \rres,\lres)$ is an algebra such that the following are satisfied:
\begin{itemize}
\item
$\le$ is a reflexive ordering on $A$,
\item 
$(A, \circ)$ is a semigroup,
\item
$\circ$ is monotonic w.r.t.\ $\le$,
%:
%\begin{equation}\label{eq:mono}
%x\le x'\text{ and }y\le y' \text { imply } x\circ y\le x'\circ y'
%\end{equation}
%for every $x,x',y,y'\in A$ and
\item
$\rres,\lres$ are the right and left residuals of $\circ$:
\begin{equation}\label{eq:residuals}
y\le x\rres z\text{ iff } x\circ y\le z\text{ iff } x\le z\lres y
\end{equation}
for every $x,y,z$ in $A$.
\end{itemize}
We denote the class of residuated semigroups by $\RS$.
%Let $\RS^\Box$ be the ``square'' subclass of $\RS$, where we additionally require that
\end{definition}

Next we define subclasses of $\RS$ that are isomorphic to algebras of binary relation and the relation $\le$ and operations $\circ,\rres,\lres$ have ``relational'' interpretations.

\begin{definition}\label{def:rrs}
A \emph{representable residuated semigroup} ${\mathcal A}=(A,\le,\circ, \rres,\lres)$ is a subalgebra of some
$(\wp(W),\subseteq, \circ_W,\rres_W,\lres_W)$ where $W$ is a transitive relation, $\subseteq$ is the subset relation and the operations are interpreted as follows:
%meet is intersection and
\begin{align}
x\circ_W y&=\{(u,v)\in W: 
(u,w)\in x \mbox{ and } (w,v)\in y \mbox{ for some } w\},\\
x\rres_W y&=\{(u,v)\in W:
\text{for every } w, (w,u)\in x\text{ implies }(w,v)\in y\},\\
x\lres_W y&=\{(u,v)\in W:
\text{for every } w, (v,w)\in y\text{ implies }(u,w)\in x\},
\end{align}
see Figure~\ref{fig:comp}.
\begin{figure}[h]
\[
\xymatrix{
&{\exists w}\ar@{-->}[ddr]^y&&&
{\forall w}\ar@{-->}[ddl]_x\ar@{-->}[ddr]^y&&&{\forall w}&\\
&{\land}& &&{\rightarrow}&&&{\leftarrow}&\\
{u}\ar@{-->}[uur]^x\ar[rr]_{x\circ_W y}&&{v}&
{u}\ar[rr]_{x\rres_W y}&&{v}&
{u}\ar[rr]_{x\lres_W y}\ar@{-->}[uur]^x&&{v}\ar@{-->}[uul]_y
}
\]
\caption{Interpretations of composition and its residuals}\label{fig:comp}
\end{figure}
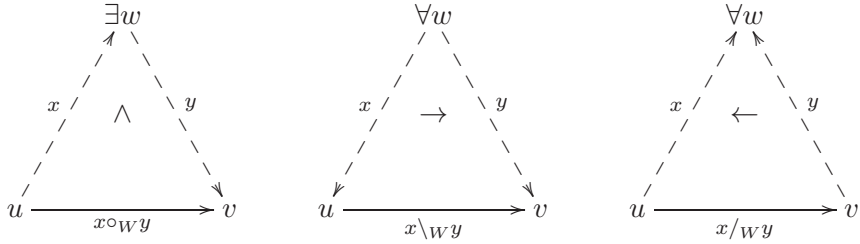
Let  $U$ be the smallest set such that $W\subseteq U\times U$. We call $U$ the \emph{base} and $W$ the \emph{unit} 
of $\mathcal A$.
We denote the class (of isomorphs) of representable residuated semigroups by $\RRS$.

Let $\RRS^\Box$ be the ``square'' subclass of $\RRS$, where we additionally require that $W=U\times U$.
\end{definition}
The subscript $W$ to the operations of representable algebras is to remind the reader that their meanings depend on the choice of the unit $W$ of the algebra. In particular, if the unit $W$ is an irreflexive relation, then $(u,u)\notin x\backslash_W x$ for any element $u$ in the base and algebra element $x$, in stark contrast to $(v,v)\in y\backslash_V y$ for any base element $v$ and algebra element $y$ for a reflexive unit $V$.

\begin{definition}\label{def:frp}
let $\sf K$ be a class of representable algebras (like $\RRS$ or $\RRS^\Box$).

We say that ${\mathcal A}\in{\sf K}$ is \emph{representable on a finite base}, or \emph{finitely representable}, if the base $U$ can be chosen to be finite in Definition~\ref{def:rrs}.

The class $\sf K$ has the \emph{finite representation property} (frp) if every finite ${\mathcal A}\in{\sf K}$ is finitely representable.
\end{definition}

\section{The main result}\label{sec-proof}
Next we show that $\RRS$ lacks the frp.

\begin{theorem}\label{thm-main}
There is a finite ${\mathcal A}\in \RRS$ that cannot be represented on a finite base.
\end{theorem}

\begin{proof}
%ake two copies of the rationals ${\mathbb Q}_0={\mathbb Q}\times\{0\}$ and ${\mathbb Q}_1={\mathbb Q}\times\{1\}$ and 
Let $U= ({\mathbb Q}\times \{0\})\cup ({\mathbb Q}\times \{1\})$ and
$$
W=\{((q,0),(r,0)):q,r\in{\mathbb Q},q<r\}\cup\{((q,1),(r,0)):q,r\in{\mathbb Q},q\le r\}
$$
where $\mathbb Q$ denotes the set of rational numbers.
Define
\begin{align}
a&=\{((q,0),(r,0)):q,r\in{\mathbb Q}, q<r\},\\
b&=\{((q,1),(q,0)):q\in{\mathbb Q}\}.
\end{align}
Let ${\mathcal A}=(A, \le,\circ_W,\backslash_W,/_W)$ be the subalgebra of $(\wp(W),\subseteq,\circ_W,\backslash_W,/_W)$ generated by $a, b$.
%We will call ${\mathbb Q_0}\cup{\mathbb Q_1}$ the \emph{base} of ${\mathcal A}$, the smallest set $U$ such that (the \emph{universe} of ${\mathcal A}$)  $W\subseteq U\times U$.
$\mathcal A$ has the following seven elements: bottom element $\emptyset$ (that we will denote by $\bot$), top element $W$ (that we will denote by $\top$), $a$, $b$,
$$
b\circ_W a=\{((q,1),(r,0)):q,r\in{\mathbb Q}, q< r\},
$$
$a\cup b\circ_W a$ (that we will denote by $a'$) and $b\cup b\circ_W a$ (that we will denote by $b'$).\footnote{Note that $\cup$ is not an operation of $\mathcal A$, but it makes sense to take the union of two elements of an algebra of binary relations.}
We include the tables for the operations (where we abbreviate $b\circ_W a$ as $ba$).

\begin{table}
\begin{center}
\begin{tabular}{| c || c | c | c | c | c | c | c |}
\hline
$\circ_W$ & $a$ & $b$ & $\bot$ & $ba$ & $\top$ & $b'$ & $a'$ \\
\hline \hline
$a$ & $a$ & $ba$ & $\bot$ & $ba$ & $a'$ & $ba$ & $a'$ \\
\hline
$b$ & $\bot$ & $\bot$ & $\bot$ & $\bot$ & $\bot$ & $\bot$ & $\bot$ \\
\hline
$\bot$ & $\bot$ & $\bot$ & $\bot$ & $\bot$ & $\bot$ & $\bot$ & $\bot$ \\
\hline
$ba$ & $\bot$ & $\bot$ & $\bot$ & $\bot$ & $\bot$ & $\bot$ & $\bot$ \\
\hline
$\top$ & $a$ & $ba$ & $\bot$ & $ba$ & $a'$ & $ba$ & $a'$ \\
\hline
$b'$ & $\bot$ & $\bot$ & $\bot$ & $\bot$ & $\bot$ & $\bot$ & $\bot$ \\
\hline
$a'$ & $a$ & $ba$ & $\bot$ & $ba$ & $a'$ & $ba$ & $a'$ \\
\hline
\end{tabular}
\end{center}
\caption{Table for composition}\label{tab:comp}
\end{table}

\begin{table}
\begin{center}
\begin{tabular}{| c || c | c | c | c | c | c | c |}
\hline
$\rres_W$ & $a$ & $b$ & $\bot$ & $ba$ & $\top$ & $b'$ & $a'$ \\
\hline \hline
$a$ & $\top$ & $b'$ & $\top$ & $b'$ & $b'$ & $b'$ & $b'$ \\
\hline
$b$ & $b'$ & $b'$ & $\top$ & $b'$ & $b'$ & $b'$ & $b'$ \\
\hline
$\bot$ & $b'$ & $b'$ & $\top$ & $b'$ & $b'$ & $b'$ & $b'$ \\
\hline
$ba$ & $b'$ & $\top$ & $\top$ & $\top$ & $b'$ & $\top$ & $b'$ \\
\hline
$\top$ & $\top$ & $\top$ & $\top$ & $\top$ & $\top$ & $\top$ & $\top$ \\
\hline
$b'$ & $b'$ & $\top$ & $\top$ & $\top$ & $b'$ & $\top$ & $b'$ \\
\hline
$a'$ & $\top$ & $\top$ & $\top$ & $\top$ & $\top$ & $\top$ & $\top$ \\
\hline
\end{tabular}
\qquad
%\end{center}
%\caption{Table for right residual}\label{tab:rres}
%\end{table}
%\begin{table}
%\begin{center}
\begin{tabular}{| c || c | c | c | c | c | c | c |}
\hline
$\lres_W$ & $a$ & $b$ & $\bot$ & $ba$ & $\top$ & $b'$ & $a'$ \\
\hline \hline
$a$ & $a$ & $\bot$ & $\bot$ & $b'$ & $\top$ & $b'$ & $\top$ \\
\hline
$b$ & $\top$ & $\top$ & $\top$ & $\top$ & $\top$ & $\top$ & $\top$ \\
\hline
$\bot$ & $\top$ & $\top$ & $\top$ & $\top$ & $\top$ & $\top$ & $\top$ \\
\hline
$ba$ & $\top$ & $\top$ & $\top$ & $\top$ & $\top$ & $\top$ & $\top$ \\
\hline
$\top$ & $a$ & $\bot$ & $\bot$ & $b'$ & $\top$ & $b'$ & $\top$ \\
\hline
$b'$ & $\top$ & $\top$ & $\top$ & $\top$ & $\top$ & $\top$ & $\top$ \\
\hline
$a'$ & $a$ & $\bot$ & $\bot$ & $b'$ & $\top$ & $b'$ & $\top$ \\
\hline
\end{tabular}
\end{center}
\caption{Tables for the residuals}\label{tab:lres}
\end{table}

Observe that
\begin{align}
a&=a\circ_W a\\
a&\not\le a\circ_W(b\backslash_W b)\circ_W a\label{eq:bb}
\end{align}
since $a$ is a transitve and dense relation, and $b\backslash_W b=b'$ and $a\circ_W b'=\bot$ whence 
$$
a\circ_W(b\backslash_W b)\circ_W a=a\circ_W b'\circ_W a=\bot\circ_W a=\bot\not\ge a
$$ 
in ${\mathcal A}$.

Next assume that ${\mathcal A}$ can be represented on a finite base.
That is, we have a finite set $U'$ and a transitive relation $V\subseteq U'\times U'$ such that $\mathcal A$ is isomorphic to
${\mathcal B}=(B, \le,\circ_V,\backslash_V,/_V)\subseteq (\wp(V),\subseteq,\circ_V,\backslash_V,/_V)$.

Let $(x,y)\in a$ in $\mathcal B$.  We show that $(x,y)\in a\circ_V(b\backslash_V b)\circ_V a$.
Since $a$ is dense ($a\le a \circ_V a$), we have $z\in U'$ (not necessarily distinct from $x,y$)
such that $(x,z), (z,y)\in a$. Continuing in the same vein (using $(x,z), (z,y)\in a\le a \circ_V a$, etc.) we can create an
$a$-path of elements $(x=z_0,z_1,\dots, z_i,\dots, z_{n-1},z_n=y)$ of unlimited length (i.e.\ each $(z_i,z_{i+1})\in a$). Since $U'$ is finite, after finitely many iterations of the process of expanding the path we get that $z_i=z_j$ for two elements of the path. Since $a$ is transitive ($a\circ_V a\le a$), we have $(x,z_i),(z_i,z_j),(z_j,y)\in a$. 
Thus $(z_i,z_j)=(z_i,z_i)\in V$, whence $(z_i,z_i)\in b\backslash_V b$.  We get that $(x,y)\in a\circ_V(b\backslash_V b)\circ_V a$, contradicting that $\mathcal B$ should satisfy \eqref{eq:bb}.
\end{proof}

\section{Conclusions}\label{sec-conc}
The above proof shows that $\RRS$ does not have the \emph{finite base property}  (fbp) w.r.t.\  implications:
the implication
\begin{equation}\label{eq:imp}
(a\le a\circ a \land a\circ a\le a) \to a\le a\circ b\backslash b\circ a
\end{equation}
is not $\RRS$-valid but cannot be refuted in any $\RRS$ with finite base.

The solution to the following problem would answer the question of the \emph{finite relational model property} (frmp) for the LC:
whether non-derivable sequents of LC are refutable in relational models with finite bases.
\begin{problem}
Does $\RRS$ have the finite base property w.r.t.\  atomic formulas $t\le s$ (where $t,s$ are $\{\circ,\backslash,/\}$-terms of variables).
\end{problem}

We can ask the same problems for the square version $\RRS^\Box$ of $\RRS$.
These problems are related to that version of the LC that allows the empty word and sequents with empty antecedents \cite{lam-cal-61}.
The corresponding $\RRS$ is a class with reflexive units $W$, thus every $x\rres_W x$ and $x\lres_W x$  subsume the identity relation,
see Andr\'eka and Mikul\'as~\cite{AM-lam-94}.
Thus the proof of Theorem~\ref{thm-main} does not work in this case.
\begin{problem}
Does $\RRS^\Box$ have the finite representation and finite base properties?
\end{problem}

\paragraph
{\textbf{Acknowledgements:}
Thanks are due to Robin Hirsch and Ian Hodkinson for useful discussions.}


\begin{thebibliography}{MMMMM}

%\bibitem[AB75]{AB-ant-75}
%{A.R. Anderson and N.D. Belnap},
%{\it Entailment. The Logic of Relevance and Necessity. Vol.\ I.},
%Princeton University Press. 1975.

\bibitem[ABD92]{ABD-ant-92}
{A.R. Anderson, N.D. Belnap and J.M. Dunn},
{\it Entailment. The Logic of Relevance and Necessity. Vol.\ II.},
Princeton University Press, 1992.

%\bibitem[An88]{And-rep-88}
%{H. Andr\'eka},
%``On the representation problem of distributive
%semilattice-ordered semigroups'',
%preprint, Mathematical Institute of the
%Hungarian Academy of Sciences, 1988.
%Abstracted in {\it Abstracts of the American Mathematical Society}, 10(2):174 
%(March 1989).

%\bibitem[An91]{And91}
%{H. Andr\'eka},
%``Representation of distributive lattice-ordered semigroups
%with binary relations'',
%{\it Algebra Universalis}, 28:12--25, 1991.

\bibitem[AM94]{AM-lam-94}
{H. Andr\'eka and Sz.\ Mikul\'as},
``Lambek calculus and its relational semantics: 
completeness and incompleteness'',
{\it Journal of Logic, Language and Information},
3:1--37, 1994.




\bibitem[vB91]{ben-lan-91}
{J. van Benthem},
{\it Language in Action},
North-Holland, 1991.

\bibitem[BDM09]{BDM-rel-09}
{K. Bimb\'o, J.M. Dunn and R.D. Maddux},
``Relevance logic and relation algebras'',
{\it Review of Symbolic Logic}, 2(1):102--131, 2009. 

%\bibitem[BS78]{BS78}
%{D.A. Bredikhin and B.M. Schein},
%``Representation of ordered semigroups and lattices by binary
%relations'',
%{\it Colloquium Mathematicum}, 39:1--12, 1978.

%\bibitem[Do92]{Do92}
%K. Do\v{s}en, 
%``A brief survey of frames for the Lambek calculus'', 
%{\em Zeitschrift f\"u{}r Mathematische Logik und Grundlagen der Mathematik},
%{\it Zeitsch.\ Math.\ Logik und Grund.\ der Math.},
%38:179--187, 1992.

%\bibitem[Du66]{Du66}
%{J.M. Dunn},
%{\it The Algebra of Intensional Logics},
%PhD dissertation, University of Pittsburgh, 1966. 

%\bibitem[Du82]{Dun82}
%{J.M. Dunn},
%``Relational representation of quasi-Boolean algebras'', 
%{\it Notre Dame Journal of Formal Logic}, 23:353--357, 1982.

%\bibitem[Du93]{Du93}
%J.M. Dunn, 
%``Partial gaggles'', 
%in K. Do\v{s}en and P. Schroeder-Heister (editors),
%\emph{Substructural Logics}, pages 63--109, Clarendon, 1993.

%\bibitem[Du92]{Dun92}
%{J.M. Dunn},
%``Non-commutative linear logic'', e-mail from 1992.
%Available at
%\verb+http://www.seas.upenn.edu/~sweirich/types/archive/1992/msg00028.html+

%\bibitem[Du01]{Dun01}
%{J.M. Dunn},
%``A representation of relation algebras using Routley--Meyer frames'',
%in C.A. Anderson and M. Zel\"eny (eds.), 
%{\it Logic, Meaning and Computation}, Kluwer, 2001, 77--108.

%\bibitem[FR90]{FoRo90}
%{J.M. Font and G. Rodriguez},
%``Note on algebraic models for relevance logic'',
%{\it Zeitschrift f\"u{}r Mathematische Logik und Grundlagen der Mathematik},
%36:535--540, 1990.

%\bibitem[FR94]{FoRo94}
%{J.M. Font and G. Rodriguez},
%``Algebraic study of two deductive systems of relevance logic'',
%{\it Notre Dame Journal of Formal Logic}, 35:369--397, 1994.

%\bibitem[DGP05]{DGP05}
%{M. Gehrke, M. Dunn and A. Palmigiano},
%``Canonical extensions and relational completeness of  some substructural logics,'' 
%\emph{Journal of Symbolic Logic}, 70(3):713--740, 2005.

%\bibitem[Hi95]{Hi94c}
%R. Hirsch,
%\newblock ``Completely representable relation algebras'',
%\newblock {\em Bulletin of the IGPL}, 3(1):77--92, 1995.

%\bibitem[Hi05]{hir:leq} 
%{R. Hirsch}, 
%``The class of representable ordered monoids has a recursively enumerable, 
%universal axiomatisation but it is not finitely axiomatisable'', 
%\emph{Logic Journal of the IGPL}, {13}:159--171, 2005.

%\bibitem[HH97]{HH97d}
%{R. Hirsch and I. Hodkinson},
%``Complete representations in algebraic logic'',
%{\it Journal of Symbolic Logic}, 62:816--847, 1997.

%\bibitem[HH00]{HH:rb}
%R. Hirsch and I. Hodkinson,
%\newblock ``Relation algebras with $n$-dimensional relational bases'',
%\newblock {\em Annals of Pure and Applied  Logic}, 
%101:227--274, 2000.

%\bibitem[HH01a]{HH:raca2}
%R. Hirsch and I. Hodkinson,
%\newblock ``Relation algebras from cylindric algebras, {II}'',
%\newblock {\em Annals of Pure and Applied  Logic}, 
%112:267--297, 2001.

%\bibitem[HH01b]{HH7}
%R. Hirsch and I. Hodkinson,
%\newblock ``Representability is not decidable for finite relation algebras'',
%\newblock {\em Transactions of the  American Mathematical Society}, 353:1403--1425, %2001.

\bibitem[HH02]{HH-rel-02}
{R. Hirsch and I. Hodkinson},
{\it Relation Algebras by Games},
{North-Holland}, 2002.

%\bibitem[HH09]{HH:srcas}
%R. Hirsch and I. Hodkinson,
%\newblock ``Strongly representable atom structures of cylindric algebras'',
%\newblock {\em Journal of Symbolic Logic}, 74:811-828, 2009.



%\bibitem[Ho93]{Hodg:mode93}
%{W. Hodges},
%\newblock {\it Model Theory},
%\newblock Cambridge University Press, 1993.



\bibitem[HM11]{HM-pos-11} 
{R. Hirsch and Sz. Mikul\'as},  
``Positive fragments of relevance logic and algebras
of binary relations'',  
\emph{Review of Symbolic Logic}, 4(1):81--105, 2011.


%\bibitem[Ho97]{Hod97}
%I.~Hodkinson,
%\newblock ``Atom structures of cylindric algebras and relation algebras'',
%\newblock {\em Annals of Pure and Applied  Logic}, 89:117--148, 1997.

%\bibitem[HM00]{HoMi00}
%{I. Hodkinson Sz.\ Mikul\'as},
%``Axiomatizability of reducts of algebras of relations'',
%{\it Algebra Universalis}, 43:127--156, 2000.

%\bibitem[JT52]{JT52}
%{B. J\'{o}nsson and A. Tarski},
%``Boolean algebras with operators, part II'',
%{\it American Journal of Mathematics}, 74:127--162, 1952.

%\bibitem[Ko07]{Kow07}
%{T. Kowalski},
%``Weakly associative relation algebras hold the key to the universe'',
%{\it Bulletin of the Section of Logic}, 36:145--157, 2007.

%\bibitem[Ko94a]{Ko94}
%{D. Kozen},
%``A completeness theorem for Kleene algebras and the algebra of regular events'', 
%{\it Information and Computation}, 110:366-390, May 1994.

%\bibitem[Ko94b]{Ko94b}
%{D. Kozen},
%``On action algebras'', 
%in J. van Eijck and A. Visser (editors),
%\emph{Logic and Information Flow}, pages 78--88, MIT Press, 1994. 

\bibitem[La58]{lam-mat-58}
{J. Lambek},
``The mathematics of sentence structure'',
{\it American Mathematical Monthly},
65:154--170, 1958.


\bibitem[La61]{lam-cal-61}
{J. Lambek}, 
``On the calculus of syntactic types'', in R. Jakobson (editor), 
{\it Structure of Language and Its Mathematical Aspects}, pages 166--178, AMS, 1961.

%\bibitem[Ly50]{Lyn50}
%{R.C. Lyndon},
%``The representation of relational algebras'',
%{\it Annals of Mathematics}, 51:707--729, 1950.
%\bibitem[Ma83]{Mad83}

%{R.D. Maddux},
%``A sequent calculus for relation algebras'',
%{\it Annals of Pure and Applied Logics}, 25:73--101, 1983.

%\bibitem[Ma06]{Ma06}
%{R.D.~Maddux},
%{\it Relation Algebras}, 
%volume 150 of {\em Studies in Logic and the Foundations of Mathematics}.
%North-Holland, 2006.

%\bibitem[Ma07]{Ma07}
%{R.D. Maddux},
%``Relevance logic and the calculus of relations'',
%paper presented at
%{\it International Conference on Order, Algebra and Logics},
%Department of Mathematics, Vanderbilt University, June 12-16, 2007.
%Abstract available at 
%\verb+http://www.math.vanderbilt.edu/~oal2007/submissions/submission_10.pdf+
%More detailed notes available at
%\verb+http://www.math.iastate.edu/maddux/talk.pdf+

\bibitem[Ma10]{mad-rel-10}
{R.D. Maddux},
``Relevance logic and the calculus of relations'',
{\it Review of Symbolic Logic}, 3(01):41--70, 2010.

%\bibitem[MPM96]{MPM96}
%{M. Marx, L. P\'olos and M. Masuch (editors)},
%{\it Arrow Logic and Multi-Modal Logic},
%CSLI Publications. 1996.

%\bibitem[RM72]{RM72}
%{R.K. Meyer and R. Routley}, 
%``Algebraic analysis of entailment I'', 
%{\it Logique et Analyse}, 15:407--428, 1972. 

%\bibitem[MDL74]{MDL74}
%{R.K. Meyer, J.M. Dunn and H. Leblanc},
%``Completeness of relevant quantification theories'',
%{\it Notre Dame Journal of Formal Logic}, 15:97--121, 1974.

%\bibitem[MR73]{MR73}
%{R.K. Meyer and R. Routley} 
%``Classical relevant logics I'',
%{\it Studia Logica}, 32:51--68, 1973.

%\bibitem[MR74]{MR74}
%{R.K. Meyer and R. Routley} 
%``Classical relevant logics II'',
%{\it Studia Logica}, 33:183--194, 1974.

%\bibitem[Mi04]{mik:04}
%{Sz.~Mikul\'{a}s},
%``Axiomatizability of algebras of binary relations'',
%in B.~L{\"{o}}we, B.~Piwinger and T.~R\"asch (editors), 
%{\it Classical and New Paradigms of Computation and their Complexity
%Hierarchies}, pages 187--205.
%Kluwer Academic Publishers, 2004.

\bibitem[Mi09]{mik-alg-09}
{Sz. Mikul\'as},
``Algebras of relations and relevance logic'',
{\it Journal of Logic and Computation}, 19:305--321, 2009.

%\bibitem[Mi11]{mik-rep-11}
%{Sz. Mikul\'as},
%``On representable ordered residuated semigroups'',
%{\it Logic Journal of the IGPL}, 19(1):233--240, 2011.

\bibitem[Mi15]{mik-low-15}
{Sz. Mikul\'as}, 
``Lower Semilattice-Ordered Residuated Semigroups and Substructural Logics'',
{\it Studia Logica}, 103:453-478, 2015.

 
%\bibitem[Mo64]{Mon64}
%{J. Monk},
%``On representable relation algebras'',
%{\em Michigan Mathematics Journal}, 11:207--210, 1964.



\bibitem[RM73]{RM-sem-73}
{R. Routley and R.K. Meyer},
``The semantics of entailment (I)'',
in H. Leblanc (editor), {\it Truth, Syntax and Modality}, pages 199--243, 
North-Holland, 1973.

%\bibitem[Sc91]{sch:91}
%{B.M. Schein},
%``Representation of subreducts of Tarski relation algebras'',
%in H. Andr\'eka, J.D. Monk and I. N\'emeti (editors), 
%{\it Algebraic Logic}, pages 621--635.
%North-Holland, 1991.

%\bibitem[Ta55]{Tar55}
%{A. Tarski},
%``Contributions to the theory of models, {I, II}'',
%{\em Proceedings Koninklijke Nederlandse Akadedemie van Wetenschappen}, 57:572--581 
%(= {\em Indagationes Mathematicae}, 16:582--588), 1954.

\end{thebibliography}
\end{document}